\documentclass[11pt]{article}
\usepackage{newcent} 
\usepackage{amsmath,amssymb,amsthm,url,color}
\usepackage{parskip}

\usepackage{hyperref}
\usepackage{graphicx}
\usepackage{epstopdf}
\def\today{\number\day\ \ifcase\month\or
	January\or February\or March\or April\or May\or June\or
	July\or August\or September\or October\or November\or December\fi
	\space \number\year}
\def\gobble#1#2{}
\def\shortdate{\number\day/\number\month/\expandafter\gobble\number\year}

\newcommand{\hypergeom}[5]{\mbox{$
_#1 F_#2\left( \! \left.
\begin{array}{c}
\multicolumn{1}{c}{\begin{array}{c} #3
\end{array}}\\[1mm]
\multicolumn{1}{c}{\begin{array}{c} #4
            \end{array}}\end{array}
\! \right| \displaystyle{#5}\right) $} }

\newcommand{\rr}{\mathbb{R}}
\newcommand{\nn}{\mathbb{N}}

\def\red#1{\color{red}{#1}}

\def\a{\alpha}
\def\b{\beta}

\newtheorem{theorem}{Theorem}[section]
\newtheorem{proposition}[theorem]{Proposition}
\newtheorem{lemma}[theorem]{Lemma}
\newtheorem{corollary}[theorem]{Corollary}
\theoremstyle{definition}

\newtheorem{remark}[theorem]{Remark}
\newtheorem{remarks}[theorem]{Remarks}

\numberwithin{figure}{section}
\numberwithin{equation}{section}
\numberwithin{table}{section}

\newcommand{\comment}[1]{}

\def\l{\lambda}

\def\beq{\begin{equation}}
\def\eeq{\end{equation}}

\definecolor{dkg}{rgb}{0,0.7,0}
\definecolor{dkr}{rgb}{0.9,0,0}
\definecolor{dkb}{rgb}{0,0,0.7}
\definecolor{purple}{rgb}{0.5,0,0.7}

\definecolor{gold}{rgb}{0.83, 0.69, 0.22}

\def\url#1{\href{#1}{#1}}
\begin{document}
\title{Interlacing of zeros from different sequences of Meixner-Pollaczek, Pseudo-Jacobi and Continuous Hahn polynomials}	
	
	\author{A.S. Jooste$^{1}$ and K. Jordaan$^{2}$ \\[2.5pt]
$^{1}$ Department of Mathematics and Applied Mathematics,\\ University of Pretoria, Pretoria, South Africa\\[2.5pt]
$^{2}$ Department of Decision Sciences,\\
University of South Africa, Pretoria, 0003, South Africa\\[2.5pt]
Email: {alta.jooste@up.ac.za};\enskip {jordakh@unisa.ac.za}}
	
	\maketitle
	
\begin{abstract} In this paper we consider interlacing of the zeros of polynomials from different sequences $\{p_n\}$ and $\{g_n\}$. In our main result we consider a mixed recurrence equation necessary for existence of a linear term $(x-A)$ so that the $(n+1)$ zeros of $(x-A)g_n(x)$ interlace with the $n$ zeros of $p_n$. We apply our result to Meixner-Pollaczek, Pseudo-Jacobi and Continuous Hahn polynomials to obtain new interlacing results for the zeros of polynomials of the same degree from different polynomial sequences. 
	\end{abstract} 
 
{\footnotesize\emph{Mathematics Subject Classification:}}

\section{Introduction}
 A sequence of real polynomials $\{p_n\}_{n=0}^\infty$, where $p_n$ is of exact degree $n\in \nn,$ is orthogonal with respect to a positive weight $w(x)$ on an interval $[a,b],$ if it satisfies
 $$\int_{a}^{b}x^jp_n(x) w(x)dx\begin{cases}=0,~\mbox{for} ~j=0,1,\dots,n-1\\\neq0,~\mbox{for} ~j=n.\end{cases}$$
Zeros of polynomials in such a sequence of orthogonal polynomials are all real, distinct and have the property that the zeros of consecutive polynomials in the sequence are interlacing (cf. \cite{chiharabook,szego}), i.e.,  if $\{x_{k,n}\}_{k=1}^n$ denotes the zeros of $p_n$, then
$$x_{1,n}<x_{1,n-1}<x_{2,n}<\cdots <x_{n-1,n}<x_{n-1,n-1}<x_{n,n}.$$ If polynomials $p_n$ and $q_n$ are of the same degree, the zeros are said to interlace if either $$x_{1,n}<y_{1,n}<x_{2,n}<y_{2,n}\cdots<x_{n,n}<y_{n,n}$$ or $$y_{1,n}<x_{1,n}<y_{2,n}<x_{2,n}\cdots<y_{n,n}<x_{n,n},$$ where $\{y_{k,n}\}_{k=1}^n$ denotes the zeros of $q_n$.

The question that arises naturally concerning any polynomial, especially when there are real zeros, is the location and behaviour of the zeros. When a polynomial has only real zeros, it is important for applications to know more about their properties, for example, interlacing and monotonicity with respect to the appropriate parameters. As an example, consider the study of the Laguerre-Pólya class of entire functions, which are functions that can be obtained as a limit (uniformly on compact subsets of C) of a sequence of polynomials with negative real roots. We refer the reader here to the recent papers \cite{MF} and \cite{ASokal}.

Zeros of polynomials, not in the same orthogonal sequence, were first considered in 1967 by Levit \cite{Levit} for Hahn polynomials. 
Levit proved that, for  $\gamma>m>1$, the zeros of Hahn polynomials $Q_m(x;\a,\b,\gamma)$, interlace with the zeros of $Q_m(x;\a,\b+1,\gamma)$,  $Q_m(x;\a-1,\b+1,\gamma)$, $Q_m(x;\a,\b,\gamma+1)$  and $Q_m(x;\a+1,\b+1,\gamma-1)$. 
In 1989, Askey \cite{Askey} proved that the zeros of Jacobi polynomials $P_n^{(\a,\b)}$ and $P_n^{(\a+1,\b)}$ interlace. Further results on the interlacing of the zeros of Jacobi polynomials from different sequences (with shifted parameters) followed in \cite{DJM}, including a proof of a conjecture by Askey that the zeros of $P_n^{(\a+2,\b)}$ and $P_n^{(\a,\b)}$ interlace. The zeros of 
Meixner, Krawtchouk, Meixner-Pollaczek and Hahn polynomials from different sequences were considered in \cite{CS,kjft,JT} and the zeros of some $q$-orthogonal polynomials were discussed in \cite{Gochhayat_et_al_2016, Jordaan_Tookos_2010,  Moak, KoepfJT_2017}.  

An investigation of interlacing of zeros from different sequences of Meixner-Pollaczek, Pseudo-Jacobi and Continuous Hahn polynomials, yielded examples that could not be proved with existing results. Hence, we consider polynomials $p_n$ and $g_n$ of degree $n$ that do not have interlacing zeros but that satisfy a mixed recurrence equation of a specific type involving a linear coefficient $(x-A)$, and prove that $A$ completes the interlacing in the sense that the zeros of $(x-A)g_n$ and $p_n$ interlace. In order to set the context and clarify the need for a new general interlacing result, we provide a short summary of previous results on interlacing of zeros of polynomials from different sequences. 

\begin{lemma} \cite[Thm 2.3]{DJ}\label{oldmain}
Consider monic polynomials $p_n(x)$ and $q_{n+1}(x)$ with $n$ real zeros $\{x_{i,n}\}_{i=1}^n$ and $n+1$ real zeros
$\{z_{i,n+1}\}_{i=1}^{n+1}$, respectively, satisfying the interlacing property \begin{equation} z_{1,n+1}<x_{1,n}<z_{2,n+1}<\dots<z_{n,n+1}<x_{n,n}<z_{n+1,n+1}\end{equation} on the (finite or infinite) interval $(c,d)$.
Fix  $k,n \in\nn$  with $k < n$ and  suppose $g_{n-k}$ is a polynomial of degree $n-k$ that satisfies
\begin{equation}\label{13}f(x)g_{n-k}(x)=D_{k}(x) p_{n}(x) + H(x) q_{n+1}(x)\end{equation}
where $f(x) \neq 0$ for $x\in(c,d)$ and $H(x)$, $D_{k}(x)$ are polynomials with deg($D_{k})=k$. Then the $n$ real, simple zeros of $D_{k} g_{n-k}$ interlace with the zeros of $q_{n+1}$ if $g_{n-k}$ and $q_{n+1}$ are co-prime.
\end{lemma}

\begin{lemma}\cite[Thm 3]{Brezinski_2004}\cite[Thm 1.1, Cor. 1.2 and 1.3]{kjft}\label{l1}
Let $q_{n+1}$ and $p_{n}$ be polynomials with $n+1$ real zeros
$z_{1,n+1}<z_{2,n+1}<\dots<z_{n+1,n+1}$ and $n$ real zeros
$x_{1,n}<x_{2,n}<\dots<x_{n,n}$ respectively satisfying the
interlacing property \begin{equation} z_{1,n+1}< x_{1,n} < z_{2,n+1} < x_{2,n} < \dots <z_{n,n+1}<x_{n,n}< z_{n+1,n+1}.\label{int}\end{equation} on the (finite or infinite)
interval $(c,d)$. Assume that a polynomial $g$ of degree $n+1$ or
$n$ satisfies the equation \begin{equation}
g(x)=D(x)p_{n}(x)+H(x)q_{n+1}(x).\label{main}\end{equation}
\begin{itemize}
    \item[(i)] If both $D(x)$ and $H(x)$
are continuous and have constant signs on $(c,\ d)$, then all the
zeros of $g$ are real and simple, the zeros of $g$ and $q_{n+1}$
interlace and the zeros of $g$ and $p_{n}$ interlace.
\item[(ii)] If $D(x)$ is continuous
and has constant sign on $(c,\ d)$, then all the zeros of $g$ are
real and simple and the zeros of $g$ and $p_{n+1}$ interlace.
\item[(iii)] If $g$ is of degree $n$ and $H(x)$ is continuous with constant sign
on $(c,\ d)$, then all the zeros of $g$ are real and simple and the
zeros of $g$ and $p_{n}$ interlace.
\end{itemize}
\end{lemma}

Lemma \ref{oldmain} applies to polynomials $g$ of degree $1,2, \dots,n-1$. The case where $g$ has degree $n$ or $n+1$ was considered in Lemma \ref{l1} and several other papers expanded on the case where $g$ is of degree $n+1$, see for example \cite{Brezinski_2004, Gochhayat_et_al_2016, Jordaan_Tookos_2010, KoepfJT_2017}. The question that remains is what happens to the zeros when the coefficient $H(x)$ changes sign on the interval $(c,d)$. We address this case in our main result (see Theorem \ref{DJboundsext}) in \S \ref{Stieltjes} and use Theorem \ref{DJboundsext} to obtain new interlacing results for zeros of Meixner-Pollaczek polynomials in \S \ref{MP1}, Pseudo-Jacobi polynomials in \S \ref{PJ} and Continuous Hahn polynomials in \S \ref{CH}.

\section{Main result}\label{Stieltjes}

Interlacing properties of zeros of polynomials from different sequences usually are derived using sign change arguments applied to polynomials written as linear combinations of polynomials with interlacing zeros. See, for example, \cite[p. 117]{Krein and Nudelman}. The next result proves the existence of an extra interlacing point for interlacing of zeros of two polynomials of the same degree when the polynomials in question satisfy a mixed recurrence equation of a specific form.   

\begin{theorem}\label{DJboundsext} Consider monic polynomials $p_n(x)$ and $q_{n+1}(x)$ with $n$ real zeros $\{x_{i,n}\}_{i=1}^n$ and $n+1$ real zeros
$\{z_{i,n+1}\}_{i=1}^{n+1}$, respectively, satisfying the interlacing property \begin{equation}\label{intz1}z_{1,n+1}<x_{1,n}<z_{2,n+1}<\dots<z_{n,n+1}<x_{n,n}<z_{n+1,n+1}\end{equation} on the (finite or infinite) interval $(c,d)$. Let
 $g_{n}$ be any monic polynomial of degree $n$ with $n$ real zeros, that satisfies, for each $n\in\nn,$ 
\begin{equation}\label{14}f(x)g_{n}(x)=D(x) p_n(x) + H(x) q_{n+1}(x)\end{equation}
where $f(x)\neq 0$ for $x\in(c,d)$. 
\begin{itemize}
    \item[(i)] If  $D(x)$, $H(x)$ are polynomials with $H(x)=x-A$, $A\in \rr$
 and $f(x)>0$, then, for each fixed $n\in\nn$,
 the $n+1$ real, simple zeros of $H(x) g_{n}(x)$ interlace with the zeros of
$p_{n}(x)$ if $g_{n}(x)$ and $p_{n}(x)$ are co-prime;
\item[(ii)] If $D(x)\neq 0$ for $x\in(c,d)$, then, for each fixed $n\in\nn$,
the zeros of $g_n(x)$ are real and simple and interlace with the zeros of $q_{n+1}(x)$ if $g_n(x)$ and $q_{n+1}(x)$ are co-prime.
\end{itemize} 

\end{theorem}
\begin{proof}
   \begin{itemize}
     \item[(i)]
    Since $p_n(x)$ and $q_{n+1}(x)$ have interlacing zeros, they do not have any common zeros and it follows from \eqref{14} that $g_{n}(x)$ and $p_n(x)$ can have at most one common zero, namely at the zero of $H(x)$. Assume that $g_{n}(x)$ and $p_n(x)$ do not have any common zeros, i.e., assume that $p_n(A)\neq0$ (or, equivalently, that $g_n(A)\neq 0)$. Let $x_{i,n}$, $i\in\{1,\dots,n\}$ denote the zeros of $p_n(x)$ and let $c=x_{0,n}$ and $d=x_{n+1,n}$. Then each of the $(n+1)$ intervals $(x_{\nu,n},x_{\nu+1,n})$, $\nu=0,1,2,\dots,n,$ contains exactly one zero of $q_{n+1}$. If $x_{i,n}$ and $x_{i+1,n}$, $i=1,2,\dots,n-1,$ are two consecutive zeros of $p_n(x)$, we have from \eqref{14}, that \begin{equation} g_{n}(x_{i,n})g_{n}(x_{i+1,n})=\frac{H(x_{i,n})H(x_{i+1,n})}{f(x_{i,n})f(x_{i+1,n})} q_{n+1}(x_{i,n})q_{n+1}(x_{i+1,n})\label{ix}\end{equation} is negative if and only if $A\notin(x_{i,n}, x_{i+1,n})$. Hence we deduce that, provided
$A\notin(x_{i,n}, x_{i+1,n})$, $g_{n}(x)$ has a different sign at
successive zeros of $p_n(x)$ and therefore has an odd number of zeros in each interval $(x_{i,n}, x_{i+1,n})$, $i=1,2,\dots,n,$ that
does not contain the point $A$. 
 \begin{enumerate}
   
\item Suppose that $A<x_{n,n}$.   Evaluating \eqref{14} at $x_{n,n}$, the greatest zero of $p_{n}(x)$, we have \begin{equation}f(x_{n,n}) g_{n}(x_{n,n})=H(x_{n,n})q_{n+1}(x_{n,n}).\label{x1}\end{equation}

Since $q_{n+1}(x_{n+1,n})>0$ (or $\lim_{x\to \infty}q_{n+1}(x)>0$ in the case when $d=x_{n+1,n}$ is infinite), and there is exactly one zero of $q_{n+1}(x)$ in the interval $(x_{n,n},x_{n+1,n})$ this implies that $q_{n+1}(x_{n,n})<0$. Since $A\,<x_{n,n}$, we have $H(x_{n,n})>0$ and it follows that the right-hand side of \eqref{x1} is negative. Hence $g_{n}(x_{n,n})<0$,  so there is at least one zero (or an odd number of zeros) of $g_{n}(x)$ in the interval $(x_{n,n},\infty)$.

\begin{enumerate}
    \item If $A<x_{1,n}$, denoting the zeros of $g_n(x)$ by $y_{i,n}$, $i\in\{1,2,\dots,n\}$, the only possibility for the arrangement of the zeros is \begin{equation}\label{i1}A<x_{1,n}<y_{1,n}<x_{2,n}<\dots<y_{n-1,n}<x_{n,n}<y_{n,n}.\end{equation}
\item If $A\in(x_{i^*,n},x_{i^*+1,n})$ for some $i^*\in \{1,\dots,n-1\}$, i.e., $x_{1,n}<A<x_{n,n}$, the
left-hand side of (\ref{ix}) is positive and hence the interval $(x_{i^*,n},x_{i^*+1,n})$ does not contain a zero of $g_{n}(x)$ or contains an even number of zeros of $g_{n}(x)$. Counting the available zeros and intervals, the only possibility for the arrangement of the zeros in this case is
\begin{align}\label{i2}y_{1,n}<x_{1,n}<y_{2,n}<\dots&<y_{i^*,n}<x_{i^*,n}<A<x_{i^*+1,n}<y_{i^*+1,n}<\nonumber\\&\dots<y_{n-1,n}<x_{n,n}<y_{n.n}.\end{align}
\end{enumerate}
\item Suppose that $A>x_{n,n}$, then $H(x_{n,n})<0$ and it follows from \eqref{x1} that $g_{n}(x_{n,n})>0$, so there are no zeros or an even number of zeros of $g_{n}(x)$ in the interval $(x_{n,n},\infty)$ containing $A$.  Counting the available zeros and intervals, the only possible arrangement for the zeros is
\begin{equation}\label{i3}y_{1,n}<x_{1,n}<y_{2,n}<\dots<x_{n-1,n}<y_{n,n}<x_{n,n}<A.\end{equation}
\end{enumerate} 
 It is clear that, in each case \eqref{i1}, \eqref{i2} and \eqref{i3}, the $n$ real zeros of $g_{n}(x)$ together with the point $A$, interlace with the zeros of $p_{n}(x)$.
 
\item[(ii)] For completeness we include a proof of this result which was stated in \cite[Cor. 1.2]{JT}, see also Lemma \ref{l1} (ii). By evaluating \eqref{14} at consecutive zeros $z_{i,n+1}$ and $z_{i+1,n+1}$, $i\in\{1,2,\dots,n\}$, of $q_{n+1}$ and multiplying the two equations we obtain
\[g_n(z_{i,n+1})g_n(z_{i+1,n+1})=\frac{D(z_{i,n+1})D(z_{i+1,n+1})}{f(z_{i,n+1})f(z_{i+1,n+1})}p_{n}(z_{i,n+1})p_{n}(z_{i+1,n+1}).\]
Since $D(z_{i,n+1})D(z_{i+1,n+1})>0$ and $f(z_{i,n+1})f(z_{i+1,n+1})>0$ for each $i\in\{1,2,\dots,n\}$ by the assumption that $D(x)$ and $f(x)$ have constant sign on $(c,\ d)$, it follows from \eqref{intz1} that \\$p_{n}(z_{i,n+1})p_{n}(z_{i+1,n+1})<0$. Hence $g_n(z_{i,n+1})g_n(z_{i+1,n+1})<0$ for each $i\in\{1,2,\dots,n\}$ which implies that $g_n(x)$ differs in sign at the consecutive zeros of $q_{n+1}(x)$, consequently $g_n(x)$ and $q_{n+1}(x)$ do not have any common zeros. Furthermore, the zeros of $g_n(x)$ are real and simple with an odd number of zeros of $g_n(x)$ between any two consecutive zeros of $p_{n+1}$. Counting the number of available zeros, we see that exactly one zero of $g_n(x)$ lies between any two consecutive zeros of $q_{n+1}(x)$.
\end{itemize}
\end{proof}
        
Note that the extra interlacing point obtained in Theorem \ref{DJboundsext} does not yield a general inner or outer bound for the extreme zeros of the polynomials $p_n$ or $g_n$ since all three cases \eqref{i1}, \eqref{i2} and \eqref{i3} are possible. In \S \ref{MP1} and \S \ref{PJ} we provide examples of polynomials whose zeros exhibit the type of interlacing proven in Theorem \ref{DJboundsext}. 

\section{Examples}

\subsection{Meixner-Pollaczek polynomials}\label{MP1} 

The monic Meixner-Pollaczek polynomials, defined by \cite[Section 9.7]{KLS}
\begin{equation}
P_n^{(\lambda)}(x;\phi)\label{MPdef}=  i^n(2\lambda)_n\left(\frac{ e^{2i\phi}}{e^{2i\phi}-1}\right)^n
\hypergeom{2}{1}{-n,\lambda+ix}{2\lambda}{1-\frac{1}{e^{2i\phi}}}
\;,
\end{equation}
are orthogonal with respect to the continuous weight $w(x)=|\Gamma(\lambda+ix)|^2\,e^{(2\phi-\pi)x}$ on the interval $(-\infty,\infty),$ for $n\in\nn,$  $\lambda>0$ and $0<\phi<\pi,$ where 
\[
\Gamma(z):=\int_0^\infty t^{z-1}e^{-t}dt
\]
denotes the Gamma function,
\[
\hypergeom{2}{1}
{\alpha_{1},\alpha_{2}}{\beta}{x}
=
\sum_{k=0}^\infty \frac
{(\alpha_{1})_{k}(\alpha_{2})_{k}}
{(\beta)_{k}}\,\frac{x^k}{k!}
\]
is the Gauss hypergeometric series and
\begin{eqnarray*}(a)_n& =&(a)(a+1)\cdots(a+n-1)~\mbox{for}~
n\geq1\\
(a)_0&=&1 ~\mbox{when}~ a\neq0\end{eqnarray*} is the Pochhammer symbol.

The recurrence equations used to prove interlacing properties can be obtained from the definition of the  Meixner-Pollaczek polynomials (\ref{MPdef}) and contiguous relations for $_2F_1$ hypergeometric functions \cite[p. 71]{rain}. It is straightforward to verify these equations by comparing coefficients of $x^n$. For the convenience of the reader we list the equations needed in this section:
\begin{align}\label{111}\left(\l^2+x^2\right) P_{n}^{(\l +1)}(x;\phi)&=\frac{\l (2 \l+n)}{2 \sin ^2 \phi}P_{n}^{(\lambda)}(x;\phi)+\left(x-\l \cot \phi \right) P_{n+1}^{(\lambda)}(x;\phi)\\
 \label{LL} P_{n}^{(\l)}(x;\phi )&=
  B(x) P_{n}^{(\l+1)}(x;\phi )+A(x) P_{n+1}^{(\l+1)}(x;\phi ) \\
\mbox{where}~~\nonumber A(x)&=\frac{4 \sin \phi  ( x \sin \phi-\l \cos \phi)}{(2 \l+n) (2 \l+n+1) }\\
 B(x)&=\frac{2\left(\l^2-\cos 2 \phi \left(\l^2+n\l+\l+x^2\right)+(n+1) x \sin 2 \phi +x^2\right)}{(2 \l+n+1) (2 \l+n) }\nonumber 
 \end{align}

\begin{proposition} \label{fin} Let $\l>0$ and $\phi\in(0,\pi)$. 
Provided that $P_n^{(\l)}\left(\l \cot \phi;\phi\right)\neq 0,$
the zeros of 
\begin{itemize}
\item[(i)] $P_n^{(\l)}(x;\phi)$ interlace with the zeros of  $\left(x-\l \cot \phi\right)P^{(\l+1)}_{n}(x;\phi)$;
\item[(ii)] $P_n^{(\l+1)}(x;\phi)$ are real and simple and interlace with the zeros of  $P_{n+1}^{(\l)}(x;\phi).$  
\end{itemize}
\end{proposition}
\begin{proof} \begin{itemize}
\item[]
\item[(i)] The result follows by applying Theorem \ref{DJboundsext}(i) to \eqref{111} with $g_n(x)=P_n^{(\l+1)}(x;\phi)$, $p_n=P_n^{(\l)}(x;\phi)$ and $H(x)=x-\l \cot \phi$. 
 \item[(ii)] This result was proved in \cite[Thm 3.1]{kjft} and also follows by applying Theorem \ref{DJboundsext}(ii) to \eqref{111} with $g_n(x)=P_n^{(\l+1)}(x;\phi)$, $p_n=P_n^{(\l)}(x;\phi)$ and $H(x)=x-\l \cot \phi$, since  
$\l^2+x^2>0$ on $\rr$ and $\frac{\l (2 \l+n)}{2 \sin ^2\phi} \neq 0$.
\end{itemize}
\end{proof}

\begin{remarks}\begin{itemize}\item[]
    \item It is not necessary to assume that $P_n^{(\l)}(x;\phi)$ and  $P^{(\l+1)}_{n-1}(x;\phi)$ are co-prime  in Proposition \ref{fin}(ii) since it follows from the equation \begin{align*} \left(\l^2+x^2\right)&(\cos{\phi}-\cos 3\phi) P_{n-1}^{(\l +1)}(x;\phi)\\&= (n x \sin 2\phi-\l-\l \cos 2\phi)P_{n}^{(\lambda)}(x;\phi)+2\l(2\l +n-1)\cos{\phi} P_{n-1}^{(\lambda)}(x;\phi)\end{align*} used in \cite{kjft} (which is not for the monic case) that there cannot be any common zeros.
\item
The point $\l\cot{\phi}$ that completes the interlacing of the zeros of $P_n^{(\l)}(x;\phi)$ and $P_n^{(\l+1)}(x;\phi)$ in Proposition \ref{fin}(i), depends only on the parameters $\l$ and $\phi$ and is independent of the degree $n\in\nn$.
\end{itemize}
\end{remarks}
Proposition \ref{fin} addresses interlacing of the zeros of Meixner-Pollaczek polynomials of the same degree and decreasing degree when the parameter $\l$ has been shifted by unity. \comment{Numerical examples show that interlacing in both of these cases breaks down for a parameter shift of $2$.} The zeros of $P_n^{(\l)}(x;\phi)$
and $P^{(\l+1)}_{n+1}(x;\phi)$, where both the degree and the parameter are increased by unity, do not interlace in general as can easily be verified with numerical examples. An interesting open question is whether there are conditions on the parameters for which interlacing between the zeros of $P_n^{(\l)}(x;\phi)$ and $P_{n+1}^{(\l+1)}(x;\phi)$ holds. See \cite{ADL}  for examples of such results for Laguerre polynomials. For the special case where $\phi=\frac{\pi}{2}$, the zeros of $P_n^{(\l)}(\frac{\pi}{2};\phi)$
and $P^{(\l+1)}_{n+1}(\frac{\pi}{2};\phi)$ do interlace in general for all $\l>0$ as we will prove in our next result.
\comment{When $\phi=\frac{\pi}{2},$ in \eqref{sp}, we obtain

\begin{equation}2 \left(\l (2 \l+n)+2 x^2\right) P_{n}^{\lambda}\left(x;\frac{\pi }{2}\right)=n(2 \l+n-1)x P_{n-1}^{\lambda}\left(x;\frac{\pi }{2}\right)+4 \left(\l^2+x^2\right) P_{n}^{\lambda+1}\left(x;\frac{\pi }{2}\right)
\end{equation}}

\begin{lemma} \label{symint} Let $\l>0$. Then the zeros of $P^{(\l)}_{n}(x,\frac{\pi}{2})$ interlace with the zeros of $P_{n+1}^{(\l+1)}(x;\frac{\pi}{2})$, provided $P^{(\l)}_{n}\left(x;\frac{\pi}{2}\right)$ and $ P^{(\l+1)}_{n+1}\left(x;\frac{\pi}{2}\right)$ are co-prime.
  \end{lemma}
\begin{proof}
 Let $\lambda>0$ and replace $\phi$ by $\frac{\pi}{2}$ in \eqref{LL} to obtain
 $$P^{(\l)}_{n}\left(x;\frac{\pi}{2}\right)=\frac{2\left(2 \l^2+n\l+\l+2 x^2\right) }{(2 \l+n)_2 }P^{(\l+1)}_{n}\left(x;\frac{\pi}{2}\right)+\frac{4 x }{(2 \l+n)_2 }P^{(\l+1)}_{n+1}\left(x;\frac{\pi}{2}\right).$$ 
Suppose $P^{(\l)}_{n}\left(x;\frac{\pi}{2}\right)$ and $ P^{(\l+1)}_{n+1}\left(x;\frac{\pi}{2}\right)$ are co-prime.
Since the coefficient of $P_{n}^{(\l+1)}(x;\frac{\pi}{2})$ does not change sign on $\mathbb{R}$ and the zeros of $P_{n}^{(\l+1)}(x;\frac{\pi}{2})$ and $P_{n+1}^{(\l+1)}(x;\frac{\pi}{2})$ interlace, it follows from Theorem \ref{DJboundsext} (ii) that the zeros of $P_{n}^{(\l)}(x;\frac{\pi}{2})$ are real and interlace with the zeros of $P_{n+1}^{(\l+1)}(x;\frac{\pi}{2})$. 
\end{proof}
\comment{\begin{remark}
We note that for $-1<\l<0$, $P_{n}^{\lambda}(x;\frac{\pi}{2})$ is quasi-orthogonal of order two and hence we have proved that the $n$ zeros of the quasi-orthogonal polynomial $P_{n}^{\lambda}(x;\frac{\pi}{2})$ interlace with the $(n+1)$ zeros of the orthogonal polynomial $P_{n+1}^{\l+1}(x;\frac{\pi}{2}).$ This implies that the $n$ zeros of the quasi-orthogonal $P_{n}^{\lambda}(x;\frac{\pi}{2}),-1<\l<0,$ {\red{ are real}}, i.e., they lie in the interval of orthogonality, which is not true in general for quasi-orthogonal polynomials \end{remark}}

\subsection{Pseudo-Jacobi polynomials}\label{PJ} 
For $a,b$ real, the monic Pseudo-Jacobi polynomials \cite{JT,KLS} are defined by 
\begin{equation*}
    P_n(x;a;b)=\frac{2^n (a+ib+1)_n} {i^n (2 a+n+1)_n}~~\hypergeom{2}{1}{-n,2 a+n+1}{a+ib +1}{\frac{1-i x}{2}}, n=0,1,2\dots.\end{equation*}
    
\comment{    \\
    &=&\frac{2^n (a+ib+1)_n} {i^n (2 a+n+1)_n}\sum_{k=0}^n \frac{(-n)_k(2a+n+1)_k}{(a+ib+1)_k} \frac{(1-ix)^k}{2^k k!}, n=0,1,2\dots.}

The following  recurrence equations will be used to prove interlacing results and bounds for the zeros of Pseudo-Jacobi polynomials:
{\small{\begin{align}\label{PJ1}
B\, P_n(x;a-1,b)=&D_2(x) \,P_n(x;a,b)-\left(x-\frac{b}{a+n}\right) P_{n+1}(x;a,b)\\
 \mbox{where}~~B=&\,\frac{(2 a+n-1) (2 a+n) \left((a+n)^2+b^2\right)}{(a+n)^2 \left(4 (a+n)^2-1\right)}~~\mbox{and}\nonumber \\
D_2(x)=&\,x^2+\frac{2 a^3+a^2 (5 n+2)+a n (4 n+3)+(n+1) \left(b^2+n^2\right)}{(a+n) (a+n+1) (2 a+2 n+1)}-\frac{b (n+1) x}{(a+n+1)(a+n)}\nonumber \end{align}}}
{\small{\begin{equation}\label{PJ4}
\left(x^2+1\right)P_{n}(x;a+1,b)=\frac{(2 a+n+1) \left((a+n+1)^2+b^2\right)}{(a+n+1)^2 (2 a+2 n+1)}P_{n}(x;a,b)+\left(x-\frac{b}{a+n+1}\right) P_{n+1}(x;a,b)
\end{equation}}}
\comment{\begin{equation}\label{PJ3}
  (1-2 a n -n) P_{n}(x;a-1,b)= \left(x-\frac{b}{a+n}\right) P_{n-1}(x;a,b)-(2 a+2 n-1) P_{n}(x;a,b)
\end{equation}}

\begin{proposition}  Let $b\in \mathbb{R}$. Then the zeros of 
\begin{itemize}
\item[(i)] $P_n(x;a,b)$ interlace with the zeros of  $(x-\frac{b}{a+n})P_{n}(x;a-1,b)$ if $a<-n$, provided that $P_n(\frac{b}{a+n};a,b)\neq 0$;
\item[(ii)] $P_n(x;a,b)$ interlace with the zeros of $(x-\frac{b}{a+n+1})P_{n}(x;a+1,b)$ if $a+1<-n$, provided that $P_n(\frac{b}{a+n+1};a,b)\neq 0$.
\item[(iii)] $P_n(x;a+1,b)$ are real and simple if $a+1<-n$ and interlace with the zeros of $P_{n+1}(x;a,b)$, provided that $P_n(\frac{b}{a+n+1};a,b)\neq 0$.
\end{itemize}
\end{proposition}

\begin{proof} \begin{itemize}
\item[]
\item[(i)] The result follows by applying Theorem \ref{DJboundsext}(i) to \eqref{PJ1} with $g_n(x)=P_n(x;a-1,b)$, $p_n=P_n(x;a,b)$ and $H(x)=x-\frac{b}{a+n}$. 
 \item[(ii)] The result follows by applying Theorem \ref{DJboundsext}(i) to \eqref{PJ4} with $g_n(x)=P_n(x;a+1,b)$, $p_n=P_n(x;a,b)$ and $H(x)=x-\frac{b}{a+n+1}$.
  \item[(iii)] The result follows by applying Theorem \ref{DJboundsext}(ii) to \eqref{PJ4} with $g_n(x)=P_n(x;a+1,b)$, $p_n=P_n(x;a,b)$ and $H(x)=x-\frac{b}{a+n+1}$ and taking into account that $x^2+1>0$ on $\rr$ and $\frac{(2 a+n+1) \left((a+n+1)^2+b^2\right)}{(a+n+1)^2 (2 a+2 n+1)}\neq 0.$ 
\end{itemize}
\end{proof}

\subsection{Continuous Hahn polynomials}\label{CH} 
Monic continuous Hahn polynomials, defined by
\begin{eqnarray}
p_n(x;a,b,c,d)=i^n\frac{(a+c)_n(a+d)_n}{(a+b+c+d+n-1)_n}
\hypergeom{3}{2}{-n,n\!+\!a\!+\!c\!+\!b\!+\!d\!-\!1,a\!+\!ix}{a+c,a+d}{1},
\end{eqnarray} are
orthogonal on the interval $(-\infty,\infty)$ with respect to the weight function
$$w(a,b,c,d,x)=\Gamma(a+ix)\Gamma(b+ix)\Gamma(c-ix)\Gamma(d-ix)$$  provided that the real parts of $a,b,c$ and $d$ are positive and $c=\bar{a}$, $d=\bar{b}$ (cf. \cite{KLS}). The sequence of polynomials satisfies the three-term recurrence equation 
\begin{align}&p_{n+2}(x;a,b,c,d)=\label{TTRRCH}\left(x-i~ C\right) p_{n+1}(x;a,b,c,d)-D~ p_{n}(x;a,b,c,d)\mbox{~~with}\\
C&=\nonumber\frac{(n+1) (b+c+n) (b+d+n)}{(a+b+c+d+2 n)_2}-\frac{(a+c+n+1) (a+d+n+1) (a+b+c+d+n)}{(a+b+c+d+2 n+1)_2}+a\mbox{~and}\\
D&=\nonumber\frac{(n+1) (b+c+n) (b+d+n)(a+c+n) (a+d+n) (a+b+c+d+n-1)}{(a+b+c+d+2 n) (a+b+c+d+2 n-1)_3}.
\end{align}
Considering the above-mentioned orthogonality conditions, we will shift parameters $a$ and $c$ (alternatively $b$ and $d$) simultaneously in order to retain the orthogonality of the polynomials. In what follows, our aim is to determine a mixed recurrence equation connecting the polynomials $p_{n}(x;a+1,b,c+1,d),$ $p_{n}(x;a,b,c,d),$ and $p_{n+1}(x;a,b,c,d).$ The parameter shifted monic continuous Hahn polynomial $p_{n}(x;a+k,b+m,c+k,d+m)$ is orthogonal on $\rr$ with respect to 
\begin{eqnarray*}
&& \Gamma(a+k+ix)\Gamma(b+m+ix)
\Gamma(c+k-ix)\Gamma(d+m-ix)\\
&&\hspace*{2cm}=\sigma_{k,m}(a,b,c,d,x)~ w(a,b,c,d,x)>0.
\end{eqnarray*} where $\sigma_{k,m}(a,b,c,d,x)=(a+ix)_k(b+ix)_m(c-ix)_k(d-ix)_m$ (cf. \cite{JNK_2017}). Then, for $p,r>0$, 
\begin{eqnarray*}
\sigma_{1,0}(a,b,c,d,x)&=&\sigma_{1,0}(p+iq,r+is,p-iq,r-is,x)\\&=&(p+iq+ix)(p-iq-ix)\\
&=& (p+i(q+x))(p-i(q+x))\\
&=& p^2+(q+x)^2,
\end{eqnarray*}
with zeros $x=-q+i~p$ or $x=-q-i~p$ and
\begin{eqnarray*}
w(a,b,c,d,x)&=& \Gamma(a+ix)\Gamma(b+ix)
\Gamma(c-ix)\Gamma(d-ix)\\
&=&\Gamma(p+iq+ix)\Gamma(r+is+ix)
\Gamma(p-iq-ix)\Gamma(r-is-ix)\\
&=&\Gamma(p+i(q+x))\Gamma(r+i(s+x))
\Gamma(p-i(q+x))\Gamma(r-i(s+x))\\
&=&\Gamma(p+i(q+x))\Gamma\overline{(p+i(q+x))}\Gamma(r+i(s+x))\Gamma\overline{(r+i(s+x))}\\
&=&\Gamma(p+i(q+x))\overline{\Gamma(p+i(q+x))}\Gamma(r+i(s+x))\overline{\Gamma(r+i(s+x))}\\
&=&\left|\Gamma(p+i(q+x))\right|^2\left|\Gamma(r+i(s+x))\right|^2
\end{eqnarray*}
  
We use the theorem of Christoffel (see \cite[Thm 2.7.1]{Ism} and \cite[Section 2.5]{szego}) to obtain the equation necessary to prove the interlacing of the zeros of $p_{n}(x;a+1,b,c+1,d)$ with the zeros of $p_{n+1}(x;a,b,c,d)$  and/or with the zeros of $p_{n}(x;a,b,c,d)$. The equation obtained is: 
\begin{align}
\nonumber U(a+ix)(c-ix)&p_{n}(x;a+1,b,c+1,d)\\
=&\left|\begin{matrix}\label{det1}
                   p_{n}(-i~\bar{a};a,b,c,d)& p_{n+1}(-i~\bar{a};a,b,c,d) & p_{n+2}(-i~\bar{a};a,b,c,d)\\
                   p_{n}(i~\bar{a};a,b,c,d)& p_{n+1}(i~\bar{a};a,b,c,d) & p_{n+2}(i~\bar{a};a,b,c,d)\\
                   p_{n}(x;a,b,c,d)& p_{n+1}(x;a,b,c,d) & p_{n+2}(x;a,b,c,d)
                    \end{matrix}\right|\end{align}                    
where \begin{align*}
U
=&\left|\begin{matrix}
                   p_{n}(-i~\bar{a};a,b,c,d)& p_{n+1}(-i~\bar{a};a,b,c,d)\\
                   p_{n}(i~\bar{a};a,b,c,d)& p_{n+1}(i~\bar{a};a,b,c,d) 
                    \end{matrix}\right|\\
=& p_{n}(-i~\bar{a};a,b,c,d)p_{n+1}(i~\bar{a};a,b,c,d) -p_{n}(i~\bar{a};a,b,c,d)p_{n+1}(-i~\bar{a};a,b,c,d)
                    \end{align*}
                  \comment{  Expanding the determinant in \eqref{det1} leads to an equation involving polynomials $p_n(x;a+1,b,c+1,d),p_{n}(x;a,b,c,d),p_{n+1}(x;a,b,c,d)$ and $p_{n+2}(x;a,b,c,d)$.}
                    
From here onwards, we will replace $a,b,c,d$ by $p+i q, r + i s, p- i q, r-i s,$ respectively, with $p,r>0$, and we will use the following notation: $p_n(x)=p_{n}(x;p+i q,r+i s,p-i q,r-i s),~p_{n+1}(-i p-q)= p_{n+1}(-i p-q;p+i q,r+i s,p-i q,r-i s)$, etc. The three-term recurrence equation  \eqref{TTRRCH} now can be written as:
\begin{align}
&p_{n+2}(x)\label{TTRR2}\\
&\nonumber =
\left(x+\frac{(n+1) (2 p+2 r-1) (q+s)+(n+1)^2 (q+s)+2 (p+r-1) (p s+q r)}{2 (n+p+r) (n+p+r+1)}\right)p_{n+1}(x)\\
&-\frac{(n+1) (n+2 p) (n+2 r) (n+2 p+2 r-1) (n+p+i q+r-i s) (n+p-i q+r+i s)}{4 (n+p+r)^2 (2 n+2 p+2 r-1) (2 n+2 p+2 r+1)}p_{n}(x) \nonumber\end{align} 
and \eqref{det1} as
\begin{align}
\nonumber U\left(p^2+(q+x)^2\right)&p_{n}(x;p+i q+1,r+i s,p-i q+1,r-i s)\\
=&\left|\begin{matrix}\label{det2}
                   p_{n}(-i p-q)& p_{n+1}(-i p-q) & p_{n+2}(-i p-q)\\
                   p_{n}(i p-q)& p_{n+1}(i p-q) & p_{n+2}(i p-q)\\
                   p_{n}(x)& p_{n+1}(x) & p_{n+2}(x)
                    \end{matrix}\right|\end{align}
When we expand \eqref{det2} and divide it by $U$, replace $p_{n+2}(x)$ by the expression given in \eqref{TTRR2} and simplify, we obtain
\begin{eqnarray}
&&\label{1}\left(p^2+(q+x)^2\right) p_n(x;p+i q+1,r+i s,p-i q+1,r-i s) \\ &=& H(x)~ p_{n+1}(x)\nonumber\\
&-&\nonumber\left(\frac{(n+1) (n+2 p) (n+2 r) (n+2 p+2 r-1) \left((n+p+r)^2+(q-s)^2\right)}{4 (n+p+r)^2 (2 n+2 p+2 r-1) (2 n+2 p+2 r+1)}-\frac{W}{U}\right) p_n(x),\mbox{~where~}\\
H(x)&=& \label{H} x+\frac{(n+1) (2 p+2 r-1) (q+s)+(n+1)^2 (q+s)+2 (p+r-1) (p s+q r)}{2 (n+p+r) (n+p+r+1)}-\frac{V}{U}\\
U&=&\label{U} p_n(-q-i p) p_{n+1}(-q+i p)-p_n(-q+i p) p_{n+1}(-q-i p),\\
V&=&\label{V} p_n(-q-i p)p_{n+2}(-q+i p)-p_n(-q+i p)p_{n+2}(-q-i p) \mbox{~~~and~}\\
W&=&\nonumber p_{n+1}(-q-i p)p_{n+2}(-q+i p)-p_{n+1}(-q+i p)p_{n+2}(-q-i p).
\end{eqnarray}

Similarly, shifting $b$ and $d$ to $b+1$ and $d+1,$ respectively, we have that
\begin{align*}\sigma_{0,1}(x)&=(b+i x)_1(d-i x)_1\\
&=(b+i x)(d-i x)\\
&=(r+i s +i x)(r+i s +i x)\\
&=r^2 +(s+x)^2,
\end{align*}
with zeros $-s-i~r$ and $-s+i~r$. In the same way as explained above, we obtain
\begin{eqnarray}
&&\label{22}\left(r^2+(s+x)^2\right)p_n(x;p+i q,r+i s+1,p-i q,r-i s+1) \\
&=&G(x)~p_{n+1}(x)\nonumber\\
&-&\nonumber\left(\frac{(n+1) (n+2 p) (n+2 r) (n+2 p+2 r-1) \left((n+p+r)^2+(q-s)^2\right)}{4 (n+p+r)^2 (2 n+2 p+2 r-1) (2 n+2 p+2 r+1)}-\frac{W1}{U1}\right) p_n(x),\mbox{~where~}\\
G(x)&=&\label{G} x+\frac{(n+1) (2 p+2 r-1) (q+s)+(n+1)^2 (q+s)+2 (p+r-1) (p s+q r)}{2 (n+p+r) (n+p+r+1)}-\frac{V1}{U1}\\
U1&=&\nonumber p_n(-s-i r) p_{n+1}(-s+i r)-p_n(-s+i r) p_{n+1}(-s-i r),\\
V1&=&\nonumber p_n(-s-i r)p_{n+2}(-s+i r)-p_n(-s+i r)p_{n+2}(-s-i r) \mbox{~~~and~}\\
W1&=&\nonumber p_{n+1}(-s-i r)p_{n+2}(-s+i r)-p_{n+1}(-s+i r)p_{n+2}(-s-i r).
\end{eqnarray}
\begin{proposition} \label{conthahn1} Let $p,r>0$. The zeros of 
\begin{itemize}
\item[(i)]$p_n(x;p+i q,r+i s,p-iq,r-is)$ interlace with the zeros of 
$$H(x)p_n(x;p+1+iq,r+is,p+1-iq,r-is),$$ where
$H(x)$ is given in \eqref{H},
provided that $p_n(x;p+1+iq,r+is,p+1-iq,r-is)$ and $p_n(x)$ do not have a zero in common; 
\item[(ii)] $p_n(x;p+i q,r+i s,p-iq,r-is)$ interlace with the zeros of 
$$G(x)p_n(x;p+iq,r+1+is,p-iq,r+1-is),$$
with
$G(x)$ given in \eqref{G}, provided that the polynomials $p_n(x;p+iq,r+1+is,p-iq,r+1-is)$ and $p_n(x)$ do not have a zero in common.
\end{itemize}
\end{proposition}
\begin{proof} \begin{itemize}
\item[]
\item[(i)] Assume that the polynomials $p_n(x)$ and $p_n(x;p+1+iq,r+is,p+1-iq,r-is)$ do not have a common zero. Since $p^2+(q+x)^2>0$ on $\rr$, the result follows by applying Theorem \ref{DJboundsext}(i) to \eqref{1} with $g_n(x)=p_n(x;p+1+i q,r+ i s,p+1-i q,r- is)$, $p_n=p_n(x;p+i q,r+i s,p-i q,r-i s)$ and 
$$H(x)=x+\frac{(n+1) (2 p+2 r-1) (q+s)+(n+1)^2 (q+s)+2 (p+r-1) (p s+q r)}{2 (n+p+r) (n+p+r+1)}-\frac{V}{U},$$
where the values of $U$ and $V$ are given in \eqref{U} and \eqref{V}, respectively.
 \item[(ii)] We assume that the polynomials under consideration are co-prime.  Since $r^2+(s+x)^2>0$ on $\rr$, the result follows by applying Theorem \ref{DJboundsext}(i) to \eqref{22} with $g_n(x)=p_n(x;p+i q,r+1+ i s,p-i q,r+1- i s)$, $p_n=p_n(x;p+i q,r+i s,p-i q,r-i s)$ and 
$H(x)=G(x)$, given in \eqref{G}.
\end{itemize}\end{proof}

As an example, when $n=5,p=2,q=1,r=4,s=3,$ the zeros of $p_n(x;2+i,4+3 i,2-i,4-3i)$ are:
$$-4.445;-2.957;-1.746,-0.601,0.750$$
and they interlace with the zeros of 
\begin{itemize}
    \item [(i)]
$(x+0.636)p_n(x;2+1+i,4+3 i,2+1-i,4-3i)$, which are:
$$-4.766,-3.205,-1.889,\mathbf{-0.636},-0.597,0.911,$$
where $-0.636$, is the zero of $H(x)$, the linear coefficient of $p_{n+1}(x)$ in \eqref{1} for the given values;
\item[(ii)]
$(x+3.727)p_n(x;2+i,4+1+3 i,2-i,4+1-3i),$ namely
$$-4.483,\mathbf{-3.727},-2.919,-1.664,-0.485,0.915$$
where $-3.727$, is the zero of $G(x)$ in \eqref{22} for the given values.
\end{itemize}

In the special case when $a,b,c$ and $d$ are real, $x=0$ is the point needed to have the desired interlacing pattern.
\begin{corollary} \label{conthahn} Let $p,r>0$. Then the zeros of 
\begin{itemize}
\item[(i)]$p_n(x;p,r,p,r)$ interlace with the zeros of $x~p_n(x;p+1,r,p+1,r)$, when $n$ is even; 
\item[(ii)]$p_n(x;p,r,p,r)$ interlace with the zeros of $x~p_n(x;p,r+1,p,r+1)$, when $n$ is even. 
\end{itemize}
\end{corollary}
\begin{proof} \begin{itemize}
\item[]
\item[(i)] 
By letting $q=s=0$ in \eqref{1}, we obtain
\begin{eqnarray}\label{ch}
   \left(p^2+x^2\right) p_{n}(x;p+1,r,p+1,r)=-A_n p_{n}(x;p,r,p,r)+x~p_{n+1}(x;p,r,p,r) 
\end{eqnarray}
where
\begin{align*}
A_n&=\frac{(n+1) (n+2 r) (n+2 p)  (n+2 p+2 r-1)}{4 (2 n+2 p+2 r+1)(2 n+2 p+2 r-1) }+\frac{p_{n+2}(i p;p,r,p,r)}{p_{n}(i p;p,r,p,r)}\\
&=\frac{(n+1) (n+2 r) (n+2 p)  (n+2 p+2 r-1)}{4 (2 n+2 p+2 r+1)(2 n+2 p+2 r-1) }-\frac{(n+2 p)_2 (n+p+r)_2 (n+2 p+2 r-1)_n}{(n+2 p+2 r+1)_{n+2}}.
\end{align*}

Let $n$ be an even number, then $p_n(0;p,r,p,r)\neq 0$ and $p_n(0;p+1,r,p+1,r)\neq 0$, since $x=0$ is a zero of all continuous Hahn polynomials with odd degree. The result follows by applying Theorem \ref{DJboundsext}(i) to \eqref{ch} with $g_n(x)=p_n(x;p+1,r,p+1,r)$, $p_n=p_n(x;p,r,p,r)$ and $H(x)=x$. 
 \item[(ii)] By letting $q=s=0$ in \eqref{22}, we obtain
 \begin{equation}\label{2}\left(r^2+x^2\right) p_n(x;p,r+1,p,r+1)=-B_n p_n(x;p,r,p,r)+x~ p_{n+1}(x;p,r,p,r),\end{equation} with
\begin{equation*}B_n=\frac{(n+1) (n+2 r)(n+2 p) (n+2 p+2 r-1)}{4(2 n+2 p+2 r+1)(2 n+2 p+2 r-1)}+\frac{p_{n+2}(i r;p,r,p,r)}{p_{n}(i r;p,r,p,r)}.\end{equation*}
Again, for $n$ even, $p_n(0;p,r,p,r)\neq 0$ and $p_n(0;p,r+1,p,r+1)\neq 0$, since $x=0$ is a zero of all continuous Hahn polynomials with odd degree. The result follows by applying Theorem \ref{DJboundsext}(i) to \eqref{2} with $g_n(x)=p_n(x;p,r+1,p,r+1)$, $p_n=p_n(x;p,r,p,r)$ and $H(x)=x$. 
  \end{itemize}
\end{proof}

As an example, when $n=6,p=2,q=0,r=4,s=0,$ the zeros of $p_n(x;2,4,2,4)$ are:
$$-3.041,-1.623,-0.511,~0.511,~1.623,~3.041$$
and they interlace with the zeros of 
\begin{itemize}
    \item [(i)]
$x~p_n(x;2+1,4,2+1,4)$, which are:
$$-3.316,-1.800,-0.575,~0,~0.575,~1.800,~3.316.$$
\item[(ii)]
$x~p_n(x;2,4+1,2,4+1),$ namely
$$-3.179,-1.694,-0.533,~0,~0.533,~1.694,~3.179.$$
\end{itemize}


\begin{thebibliography}{99}





\bibitem{ADL}J. Arvesú, K. Driver and L. Littlejohn,  Interlacing of zeros of Laguerre polynomials of equal and consecutive degree, {\it Integr. Transf. Spec. F.} 32 (5-8) (2021), 346-360.



\bibitem{Askey} R. Askey,  Graphs as an aid to understanding special functions, {\it Asymptotic and Computational Analysis,
Winnipeg 1989, Lecture Notes in Pure and Applied Mathematics} 124 (1990), 3-33.



\bibitem{Brezinski_2004} C. Brezinski, K. A. Driver, \& M. Redivo-Zaglia, Quasi-orthogonality with applications to some families of classical orthogonal polynomials, {\it Appl. Numer. Math.} 48 (2) (2004), 157--168.

\bibitem{chiharabook} T.S. Chihara, \emph{An introduction to orthogonal polynomials}, Gordon and Breach, 1978.


\bibitem{CS}L. Chihara and D. Stanton, Zeros of generalized Krawtchouk polynomials, \emph{J. Approx. Theory} 60 (1990), 43--57.













\bibitem{DJ} K. Driver and K. Jordaan, Bounds for extreme zeros of some classical orthogonal polynomials, {\it J. Approx. Theory} 164 (2012), 1200-1204.


\bibitem{DJM} K. Driver, K. Jordaan and N. Mbuyi, Interlacing of the zeros of Jacobi polynomials with different parameters, {\it Numer. Algorithms} 49 (2008), 143-152.



\bibitem{Gochhayat_et_al_2016}
P. Gochhayat, K. Jordaan, K. Raghavendar and A. Swaminathan, Interlacing properties and bounds for zeros of $_2\phi_1$ hypergeometric and little $q$-Jacobi polynomials, {\it Ramanujan J.}   40 (2016), 45--62.




\bibitem{Ism} M.E.H. Ismail, {\it Classical and quantum orthogonal polynomials in one variable}, Encyclopedia of Mathematics and its Applications  98.  Cambridge: Cambridge University Press (2005)







\bibitem{JNK_2017} A. Jooste, P. Njionou Sadjang and W. Koepf, Inner bounds for the extreme zeros of $_3F_2$ hypergeometric polynomials, {\it Integr. Transf. Spec. F.}  28 (5) (2017), 361--373.

\bibitem{kjft} K. Jordaan and F. To\'{o}kos, Interlacing theorems for the zeros of some orthogonal polynomials from different sequences, \emph{Appl. Numer. Math.} 59 (2009), 2015--2022.
 
\bibitem{Jordaan_Tookos_2010}
K. Jordaan and F. To\'{o}kos,  Mixed recurrence relations and interlacing of the zeros of some $q$-orthogonal polynomials from different sequences, {\it Acta. Math. Hungar.}  128  (1-2) (2010), 150--164.

\bibitem{JT} K. Jordaan and F. To\'okos, Orthogonality and asymptotics of Pseudo-Jacobi polynomials for non-classical parameters, \emph{J.\ Approx.\ Theory} 178 (2014), 1--12.











\bibitem{KLS} R. Koekoek, P.A. Lesky and R.F. Swarttouw, \emph{ Hypergeometric orthogonal polynomials and their $q$-analogues}. Springer Monogr. Math., Springer Verlag, Berlin (2010).






\bibitem{Krein and Nudelman}
M.G. Krein and A.A. Nudelman, {\it The Markov Moment Problem and Extremal Problems}. Translations of Mathematical Monographs, vol. 5, American Mathematical Society, 1977.

\bibitem{Levit} R.J. Levit, The zeros of the Hahn polynomials, {\it SIAM Review} 9 (2) (1967), 191-203.

\bibitem{MF} A. Martinez-Finkelshtein, R. Morales and D. Perales, Real Roots of Hypergeometric Polynomials via Finite Free Convolution, {\it Int. Mat. Res. Not.} 2024 (16) (2024), 22642 - 22687.

\bibitem{Moak}
D.S. Moak, The $q$-Analogue of the Laguerre Polynomials, {\it J. Math. Anal. Appl.} 81 (1981), 20--47.



\bibitem{rain} E.D. Rainville, \emph{Special Functions}, The Macmillan Company, New York, 1960.




\bibitem{ASokal} A.D. Sokal. When does a hypergeometric function $_pF_q$ belong to the Laguerre-Pólya class $LP^+$? {\it J. Math. Anal. Appl.} 515  (2022), 126432.

\bibitem{szego} G. Szeg\H{o}, \emph{Orthogonal Polynomials}, American Mathematical Society, New York, 1959.

\bibitem{KoepfJT_2017}
D.D. Tcheutia, A.S. Jooste and W.  Koepf,  Mixed recurrence relations and interlacing properties  for zeros of $q$-classical orthogonal  polynomials, \textit{ Appl. Numer. Math.} 125 (2018), 86--102.


\end{thebibliography}
\end{document}